\documentclass[a4paper,oneside,12pt]{article}

\usepackage{amsmath,amsfonts,amscd,amssymb}
\usepackage{longtable,geometry}
\usepackage[english]{babel}
\usepackage[utf8]{inputenc}
\usepackage[active]{srcltx}
\usepackage[T1]{fontenc}
\usepackage{graphicx}
\usepackage{pstricks}
\usepackage{bbm}
\usepackage{mathtools}
\usepackage{hyperref}

\usepackage{MnSymbol}
\usepackage{stmaryrd}
\usepackage{nicefrac}
\usepackage{calrsfs}

\usepackage{xcolor}
\usepackage{framed}

\colorlet{shadecolor}{blue!15}

\usepackage{amsthm}
\newtheorem{theorem}{Theorem}

\newtheorem{claim}{Claim}
\newtheorem*{claim*}{Claim}

\newtheorem{proposition}[theorem]{Proposition}

\newtheorem{remark}[theorem]{Remark}

\newcommand{\bbN}{\mathbb{N}}

\newcommand{\bbR}{\mathbb{R}}

\newcommand{\bbZ}{\mathbb{Z}}

\numberwithin{equation}{section}

\newcommand{\rk}[1]{\bgroup\color{red}%
  \par\medskip\hrule\smallskip%
  \noindent\textbf{#1}%
  \par\smallskip\hrule\medskip\egroup}

\title{A note on continuity of magnetization at criticality for the ferromagnetic Ising model on amenable quasi-transitive graphs with exponential growth}
\author{Aran Raoufi}
\date{\today}

\begin{document}
\maketitle
 
 \begin{abstract}
The purpose of this modest note is to point out that the proof of the recent result of Huchcroft \cite{hutchcroft2016critical} concerning continuity of phase transition in Bernoulli percolation is applicable to the setting of the Ising model with free boundary condition. This observation, together with a recent result of Aizenman, Duminil-Copin, and Sidoravicius \cite{AizDumSid15} implies that  the Ising model on  amenable quasi-transitive  graph with exponential growth undergoes a second order phase transition. 
 \end{abstract}
 \section{Introduction} 
Let $G=(V(G),E(G))$ be a countable locally finite graph. Let $(J_{xy})_{x,y\in V(G)}$ be a family of nonnegative real numbers which are invariant under automorphisms of G, that is for any automorphism $\tau$ of $G$, $J_{\tau(x) \tau(y)} = J_{xy}$.
For $h \in \bbR$, the ferromagnetic Ising model on a finite subset $\Lambda \subset V(G)$ is defined by the Hamiltonian
$$H_{\Lambda,h}(\sigma) = -\sum_{x,y \in \Lambda} J_{xy} \sigma_x \sigma_y  -  h \sum_{x \in \Lambda} \sigma_x $$
for any $\sigma \in \{-1, +1\}^\Lambda$. 

For $\beta \in [0, \infty)$, define the Ising measure on $\Lambda$ magnetic field $h$ at inverse temperature $\beta$ to be the measure $ \mu_{\Lambda, \beta, h}$ defined on configurations $\sigma \in  \{-1, +1\}^\Lambda$ by 
$$ \mu_{\Lambda, \beta, h} (\sigma) = \frac {\exp (-\beta H_{\Lambda,h}(\sigma))} {Z(\Lambda, \beta,h)},$$
where $Z(\Lambda, \beta, h)$ is a normalizing constant defined in such a way that the total mass of the measure is equal to one. 

Let $(\Lambda_n)_{n \geq 1}$ be a sequence of nested finite subgraphs exhausting $G$. Define the Ising measure on $G$ at inverse temperature $\beta$ with external field $h$ to be the weak limit of measures $\mu_{\Lambda_n,\beta, h}$, and denote this measure by $\mu_{G, \beta, h}$. When $h=0$ we denote the measure by $\mu_{G, \beta}^0$. Let $\mu_{G, \beta}^+$ denote the measure which is the weak limit of the measures $\mu_{G,\beta, h}$ as $h \rightarrow 0^+$. $\mu_{G, \beta}^0$ (resp. $\mu_{G, \beta}^+$) is called Ising measure with free (res. plus) boundary condition.

Fix a vertex $x \in V(G)$. The quantity $\mu_{G, \beta}^+ (\sigma_x)$ is of special interest and is called spontaneous magnetization. The critical point of the model is defined as
$$ \beta_c := \inf \{ \beta: \mu_{G, \beta}^+ (\sigma_x) > 0 \}.$$
It is easy to see that the value of $\beta_c$ is independent from the choice of the vertex $x$. 

Whether spontaneous magnetization is continuous in $\beta$ or not is a natural mathematical and physical question.
Determining continuity of magnetization on $\bbZ^d$ has a long-standing history. 
The right continuity of spontaneous magnetization could be easily shown by simple semi-continuity arguments. (See Claim 3 of Proposition \ref{propj} below for an example of this type of arguments.) Therefore, the main question is left continuity of the spontaneous magnetization.

In dimension $d=2$, Yang \cite{Yan52}, inspired by works of Onsager \cite{ONS44} and Kaufmann \cite{kaufman1949crystal}, computed the spontaneous magnetization for the nearest neighbor case and continuity for all $\beta$ was established. 
In dimension $d\geq 4$ continuity at $\beta_c$ was proved in \cite{AizFer86} at criticality for nearest neighbor case and more generally reflection positive models. 
(See also \cite{sakai2007lace}.) 
In \cite{bodineau2006translation}, the continuity of magnetization was settled for any $\beta> \beta_c$ for $d\geq 3$. Finally, in \cite{AizDumSid15} the continuity was established in $d=3$ at criticality for nearest neighbor and reflection positive models.

Apart from $\bbZ^d$, this question has been studied on regular trees in \cite{haggstrom1996random} where continuity at $\beta_c$ is proved. (See also \cite{bleher1995purity} and references therein.) Not much is known about continuity of magnetization for general graphs.

Before stating our main theorem we recall the definitions of amenability and exponential growth for graphs. 
Let $G$ be  a countable locally finite graph. Let $d(.,.)$ denote the graph distance on $G$.
For a vertex $x \in V(G)$, define $\Lambda_n(x) := \{ y \in V(G): d(x,y) \leq n\}.$
Fix a vertex $x \in V(G)$. We say $G$ has \emph{exponential growth} if $$\liminf_{n \rightarrow \infty} \vert \Lambda_n(x) \vert ^{1/n} > 1.$$
It is easy to see that the above definition is independent of the choice of the vertex $x$. For a subset $A \subset V(G)$ define $\partial A := \{ x \in A : \exists y \in G \setminus A, \{x, y \} \in E(G) \}$. 
We say $G$ is \emph{amenable} if 
$$\inf_{A \subset G, \: \vert A \vert < \infty} \frac{\vert \partial A \vert}{\vert A \vert} = 0.$$

\begin{theorem} \label{thm:main}
Let $G$ be amenable quasi-transitive graph with exponential growth. 
The magnetization is continuous in $\beta$ at $\beta_c$ i.e., for any $x \in G$,
$$\mu_{G,\beta_c}^+ (\sigma_x) = 0.$$
\end{theorem}

Note that the only assumption on coupling constants is their invariance under automorphisms of the graph. We write the proof for transitive graphs. The proof works the same for quasi-transitive graphs (the proof of Proposition \ref{propj} follows almost the same line, the proof of \cite{duminil2015new} works the same for quasi-transitive graphs, and the setting of Theorem \ref{thm:adcs} can be lifted to amenable quasi-transitive graphs).

\paragraph{Notation.} From now on, we fix $G=(V(G), E(G))$ a countable locally finite transitive amenable graph with exponential growth. We drop $G$ from the notation. 
 
\section{Proof of Theorem 1}
We first rewrite the beautiful proof of Hutchcroft \cite{hutchcroft2016critical} in the case of Ising model with free boundary condition. 
\begin{proposition} \label{propj}
Fix $x \in V(G)$. There exists $\rho<1$ such that for all $n\geq 0$,
\begin{equation} \label{eq.prop}
\min \{ \mu_{\beta_c}^0 (\sigma_x \sigma_y) : y \in \Lambda_n(x)  \} \leq \rho ^n .
\end{equation}
\end{proposition}
\begin{proof}
Define 
$$\kappa_\beta(n) :=  \min \: \{ \mu_\beta^0 (\sigma_x \sigma_y) : y \in \Lambda_n(x)  \}.$$
The proof consists of three claims.

\begin{claim}
The sequence $(\kappa_\beta(n))_{n\geq 0}$ is supermultiplicative.
\end{claim}
\begin{proof} [Proof of Claim 1.]
Let $y \in \Lambda_{m+n}(x)$, there exists a vertex $z \in \Lambda_n(x)$ such that $y \in \Lambda_m(z)$. Griffiths' inequality \cite{Gri67} implies
$$\mu_\beta ^0 (\sigma_x \sigma_y) = \mu_\beta ^0 (\sigma_x \sigma_z \sigma_z \sigma_y) \geq  \mu_\beta ^0 (\sigma_x \sigma_z) \mu_\beta ^0(\sigma_z \sigma_y) \geq \kappa_\beta (n) \kappa_\beta (m).  $$
Thus,
$$ \kappa_\beta(n+m) =  \min \left\lbrace \mu_\beta ^0 (\sigma_x \sigma_y),  y \in \Lambda_{n+m}(x) \right\rbrace \geq \kappa_\beta(n) \kappa_\beta(m).$$
\end{proof}

\begin{claim}
There exists $\rho<1$, such that for any $\beta < \beta_c,$ 
$$\sup_{n\geq 0} \left( \kappa_\beta (n) \right)^{1/n}  < \rho.$$
\end{claim}
\begin{proof}[Proof of Claim 2.]
Based on the definition of $\kappa_\beta (n)$,
\begin{equation}\label{eq.1} 
\kappa_\beta (n)  \cdot \vert \Lambda_n \vert \leq \sum_{y \in \Lambda_n(x)} \mu_{\beta}^0 (\sigma_x \sigma_y) \leq \sum_{y \in V(G)} \mu_\beta ^0 (\sigma_x \sigma_y).
\end{equation}
An important ingredient here is the following theorem.
\begin{theorem}[\cite{AizBarFer87}, \cite{duminil2015new}] \label{finitesus}
Let $G$ be a transitive graph and $x \in V(G)$. For any $\beta < \beta_c$, 
$$\sum_{y \in V(G)} \mu_\beta ^0 (\sigma_x \sigma_y) < \infty.$$
\end{theorem}

As  $\kappa_\beta(n)$ is supermultiplicative, Fekete's lemma implies that $\lim_{n \rightarrow \infty} \left( \kappa_\beta (n) \right)^{1/n}$ exists and is equal to $\sup_n \left( \kappa_\beta (n) \right)^{1/n}$. Combining this fact with \eqref{eq.1} gives us
$$\sup_n \left( \kappa_\beta (n) \right)^{1/n} =  \lim_{n \rightarrow \infty} \left( \kappa_\beta (n) \right)^{1/n} \leq \lim_{n \rightarrow \infty} ( \frac{\sum_{y \in V(G)} \mu_\beta^0 (\sigma_x \sigma_y) }{\vert \Lambda_n \vert}) ^{1/n} =  \lim_{n \rightarrow \infty} \left( \frac{1}{\vert \Lambda_n \vert} \right) ^{1/n}.$$
As the graph has exponential growth, $\lim_{n \rightarrow \infty} ( \frac{1}{\vert \Lambda_n \vert}) ^{1/n}< 1$, and the claim follows. Note that this is the only place where exponential growth of the graph is used.
\end{proof}

\begin{claim}
The map $\beta \rightarrow \sup_n \left( \kappa_\beta (n) \right)^{1/n}$ is left continuous at $\beta \in [0, \infty)$.
\end{claim}
\begin{proof}[Proof of Claim 3.]
By Griffiths' inequality, $\mu_\beta^0(\sigma_x \sigma_y) = \sup_k \mu_{\Lambda_k, \beta, 0}^0(\sigma_x \sigma_y)$. 
Since $\Lambda_k$ is finite, and because of Griffiths' inequality,  $\beta \rightarrow \mu_{ \Lambda_k, \beta, 0}^0(\sigma_x \sigma_y)$ is continuous and increasing. The supremum of increasing continuous functions is left continuous. 
Hence for fixed $x,y \in V(G)$ the map $\beta \rightarrow \mu_\beta^0(\sigma_x \sigma_y)$ is left continuous. 

Now fix $n \in \bbN$. Since $\Lambda_n$ is finite, $\kappa_\beta(n)$ is the minimum of finitely many left continuous increasing functions, so is left continuous and increasing in $\beta$. Finally, the map $\beta \rightarrow \sup_n (\kappa_\beta(n) )^{1/n}$ is the supremum of left continuous increasing functions, so is left continuous. 
\end{proof}
Claim 2 and Claim 3 together conclude the proof. Indeed, $\sup_n (\kappa_\beta(n))^{1/n}$ is uniformly bounded above by some $\rho<1$ when $\beta < \beta_c$, and since $\beta \rightarrow \sup_n (\kappa_\beta(n))^{1/n}$ is left continuous, it follows that $\sup_n (\kappa_{\beta_c} (n))^{1/n} \leq \rho$. Hence for any $n \geq 1$, $\kappa_{\beta_c} (n) \leq \rho^n$ which is the claim.
\end{proof}

Recently Aizenman, Duminil-Copin, and Sidoravicius proved the following theorem, which establish a connection between $\mu_{\beta}^+$ and $\mu_{\beta}^0$. Their approach is based on the random current representation of the Ising model.
\begin{theorem}[\cite{AizDumSid15}] \label{thm:adcs}
Let $G$ be an amenable transitive graph, and let $\beta \in [0, \infty)$. If 
\begin{equation} \label{eq:C2}
\inf_{K \subset V(G), \: \vert K \vert < \infty}  \frac{\sum_{x,y \in K}  \mu_\beta ^0 (\sigma_x \sigma_y)}{\vert K \vert^2} = 0,
\end{equation}
then for any $x \in V(G)$, $\mu_{\beta}^+ (\sigma_x) = 0$.
\end{theorem}

\begin{remark}
Theorem \ref{thm:adcs} is not stated in the above form the in \cite{AizDumSid15}. There, it is stated that  on $\bbZ^d$ if the Long Range Order parameter vanishes, then spontaneous magnetization is also $0$. However, their proof works for any amenable graph and the condition of vanishing magnetization could be weakened to \ref{eq:C2}. It is worth highlighting that amenability has a vital importance in the argument of \cite{AizDumSid15} to obtain uniqueness of infinite volume random current infinite cluster via a Burton-Keane type argument.
\end{remark}

\begin{proof}[Proof of Theorem \ref{thm:main}]
Theorem \ref{thm:adcs} implies that in order to conclude the proof of Theorem \ref{thm:main} it is enough to construct a family $\{K_n\}_{n \geq 1}$ of finite subsets of $V(G)$, such that
$$\inf_n  \frac{\sum_{x,y \in K_n}  \mu_{\beta_c} ^0 (\sigma_x \sigma_y)}{\vert K_n \vert ^2} = 0.$$

Let $c= \min_{\{ x,y\} \in E(G)} \mu_{\beta_c}^0(\sigma_x \sigma_y)$. Griffiths' inequality implies that that for any $x , y \in V(G)$, $\mu_{\beta_c}^0(\sigma_x \sigma_y) \geq c^{d(x,y)}$. 
Choose $k\geq 2$ an integer such that $c \geq \rho ^k$, Where $\rho$ is the same constant as of \eqref{eq.prop}.
Fix a vertex $x_1 \in V(G)$, and for each $n>1$, define $x_n \in V(G)$ such that $d(x_n, x_1) = k^n$ and 
$$ \mu_{\beta_c}^0(\sigma_{x_1} \sigma_{x_n}) \leq \rho ^{k^n},$$
Proposition \ref{propj} guarantees the existence of $x_n$. Let $K_n = \{x_i: 1 \leq i  \leq n \}$. For $1 < i < j \leq n$, Griffiths' inequality implies
\begin{align*}
\mu_{\beta_c}^0(\sigma_{x_i} \sigma_{x_j}) 
\leq \frac{\mu_{\beta_c}^0(\sigma_{x_1} \sigma_{x_j})}{\mu_{\beta_c}^0(\sigma_{x_1} \sigma_{x_i})}
\leq \frac{\rho ^{k^j}}{c ^{k^i}} 
\leq \rho ^{k^j-k^{i+1}}
\leq \rho ^{j-i},
\end{align*}
for $j-i$ large enough. This implies that there exists a constant $C $ independent of $n$ such that,
\begin{align*}
 \frac{\sum_{x,y \in K_n}  \mu_{\beta_c} ^0 (\sigma_x \sigma_y)}{\vert K_n \vert ^2} 
\leq \frac{C \vert K_n \vert   }{\vert K_n \vert ^2}  
\leq \frac{C}{ n}.
\end{align*}
Hence $\inf_{n \geq 1}  \frac{\sum_{x,y \in K_n}  \mu_\beta ^0 (\sigma_x \sigma_y)}{\vert K_n \vert^2} = 0$.
\end{proof}

\paragraph{Acknowledgments} The author is thankful to Hugo Duminil-Copin for reading the manuscript.
This research was supported by the NCCR SwissMAP, the ERC AG COMPASP, and the Swiss NSF.

\bibliographystyle{alpha}
\bibliography{biblicomplete}
\small\begin{flushright}
\textsc{D\'epartement de Math\'ematiques}\\
 \textsc{Universit\'e de Gen\`eve} \\
\textsc{Gen\`eve, Switzerland} \\
\textsc{E-mail:} \texttt{aran.raoufi@unige.ch}
\end{flushright}
\end{document}